\documentclass[12pt]{article}

\usepackage{setspace}
\usepackage{a4wide}

\usepackage{amsthm, amssymb, amstext}
\usepackage[fleqn]{amsmath}
\usepackage{latexsym}
\usepackage[dvips]{graphicx}
\usepackage{comment}
\usepackage{mathtools}
\usepackage{enumerate} 

\usepackage{todonotes}
\usepackage{comment}

\newtheorem{theorem}{Theorem}

\newtheorem{corollary}{Corollary}
\newtheorem{conjecture}{Conjecture}

\newtheorem{claim}{Claim}

\theoremstyle{definition}

\usepackage{amsthm}

\title{Three remarks on $\mathbf{W_2}$ graphs}
\author{Carl Feghali\thanks{Univ Lyon, EnsL, CNRS, LIP, F-69342, Lyon Cedex 07, France, email: \texttt{carl.feghali@ens-lyon.fr}} \\ \and Malory Marin\thanks{EnsL, LIP, F-69342, Lyon Cedex 07, France, email : \texttt{malory.marin@ens-lyon.fr}}}

\date{}

\begin{document}
\maketitle

\begin{abstract}
Let $k \geq 1$. A graph $G$ is $\mathbf{W_k}$ if for any $k$ pairwise disjoint independent vertex subsets $A_1, \dots, A_k$ in $G$, there exist $k$ pairwise disjoint maximum independent sets $S_1, \dots, S_k$ in $G$ such that $A_i \subseteq S_i$ for $i \in [k]$. Recognizing $\mathbf{W_1}$ graphs is coNP-hard, as shown by Chv\'atal and Slater (1993) and, independently, by Sankaranarayana and Stewart (1992). Extending this result and answering a recent question of Levit and Tankus, we show that recognizing $\mathbf{W_k}$ graphs is coNP-hard for $k \geq 2$. On the positive side, we show that recognizing $\mathbf{W_k}$ graphs is, for each $k\geq 2$, FPT parameterized by clique-width and by tree-width. Finally,  we construct graphs $G$ that are not $\mathbf{W_2}$ such that, for every vertex $v$ in $G$ and every maximal independent set $S$ in $G - N[v]$, the largest independent set in $N(v) \setminus S$ consists of a single vertex, thereby refuting a conjecture of Levit and Tankus. 
 \end{abstract}

\section{Introduction}

In a graph, an \emph{independent set} is a set of pairwise non-adjacent vertices. An independent set is said to be \emph{maximum} if it is of maximum size, and \emph{maximal} if it is not a subset of any other independent set. 

For a positive integer $k$, a graph $G$ is $\mathbf{W_k}$ if for any $k$ pairwise disjoint independent vertex subsets $A_1, \dots, A_k$ in $G$, there exist $k$ pairwise disjoint maximum independent sets $S_1, \dots, S_k$ in $G$ such that $A_i \subseteq S_i$ for $i \in [k]$ \cite{staples1979some}. A graph that is $\mathbf{W_1}$ is also commonly called \emph{well-covered}. Phrased differently, a graph is \emph{well-covered} if every maximal independent set of the graph is also maximum. Well-covered graphs have been extensively studied; see \cite{brown, chvatal, finbow, plummer,  sank} for some examples. 

 Chv\'atal and Slater \cite{chvatal} and, independently, Sankaranarayana and Stewart \cite{sank} demonstrated that recognizing well-covered (or $\mathbf{W_1}$) graphs is coNP-complete. However, the complexity of recognizing $\mathbf{W_k}$ graphs for each $k \geq 2$ has remained open, despite this class of graphs being introduced around the same time as the class of well-covered graphs (see the survey by Plummer \cite{plummer} for some references), and this was raised explicitly by Levit and Tankus \cite{levit}. It is worthwhile noting that a graph $G$ is $\mathbf{W_2}$ if and only if $G$ is $1$-well-covered (that is, a well-covered graph that remains well-covered after the removal of any one vertex) and has no isolated vertices \cite{staples1979some}. The class of $1$-well-covered graphs has received considerable attention; see for example \cite{hartnell, levit1}. 

Our first contribution is to answer this question.

\begin{theorem}\label{thm:wk}
For each $k \geq 2$, recognizing $\mathbf{W_k}$ graphs is coNP-hard. 
\end{theorem}

We prove Theorem \ref{thm:wk} in Section \ref{sec:wk}.

In our next contribution, we consider the problem of recognizing $\mathbf{W_k}$ graphs from a parameterized perspective. An instance of a \emph{parameterized problem} is a pair $(x, k)$, where $x$ is a string encoding the input and $k$ is a parameter. A parameterized problem is said to be \emph{fixed-parameter-tractable} (FPT) if it can be solved in $f(k) \cdot |x|^{O(1)}$ for some computable function $f$. 

 In the 90's, Courcelle proved that if a graph problem can be formulated in \textit{monadic second-order logic} (MSO${}_{1}$), then it is FPT when parameterized by clique-width \cite{Courcelle1,Courcelle2,Courcelle3,Courcelle5,Courcelle6}. 
Furthermore, Courcelle \cite{Courcelle4} showed that every graph problem definable in LinEMSOL (an extension of MSO${}_1$ with the possibility of optimization with respect to some linear evaluation function) is also FPT when parameterized by \textit{clique-width}. We remark that the clique-width of a graph $G$ is defined as the minimum number of labels required to construct G using only the following four operations. 
\begin{itemize}
\item Creating a new vertex with label $i$; 
\item taking the disjoint union of two labeled graphs;
\item connecting all the vertices having label $i$ to all the vertices having label $j$ with $i \neq j$; 
\item renaming the vertices with label $i$ to label $j$.
\end{itemize}
We also note that deciding whether or not a given graph has bounded clique-width is NP-complete \cite{fellows2009clique}.

Now, using Courcelle's theorem, Alves \emph{et al.} \cite{Alves} showed that recognizing well-covered graphs is FPT when parameterized by clique-width. In our next contribution, we extend this result to $\mathbf{W_k}$  graphs via the same technique. 

\begin{theorem}\label{thm:FPT}
For each $k\geq 2$, recognizing $\mathbf{W_k}$ graphs is FPT when parameterized by clique-width.
\end{theorem}

We prove Theorem \ref{thm:FPT} in Section \ref{sec:fpt}. 

Let us note that Courcelle's first theorem was about MSO${}_2$ logic (which allows quantification over edges) and about \emph{tree-width} instead of clique-width \cite{Courcelle3}.  Since it is well-known that any graph with bounded clique-width has bounded tree-width, this immediately yields

\begin{corollary}
For each $k\geq 2$, recognizing $\mathbf{W_k}$ graphs is FPT when parameterized by tree-width.
\end{corollary}

We remark that  the notion of tree-width was introduced by Robertson and Seymour \cite{robertson1986graph}. Let $G=(V,E)$ be a graph. A \textit{tree decomposition} of $G$ is a tree $T$ whose nodes $X_t$ are subsets of $V$ and such that
\begin{itemize}
\item for each edge $(u,v)\in E$, there exists a node $t\in V(T)$ such that $u,v\in X_t$, and
\item for each vertex $v\in V(G)$, the set of nodes $t$ of $V(T)$ such that $X_t$ contains $v$ induces a nonempty connected subtree of $T$.
\end{itemize}
The width of a tree decomposition of $G$ is $\max_{t\in V(T)} |X_t|-1$, and the tree-width of $G$ is the minimum width over all tree decompositions of $G$.

As our final contribution, we refute, in Section \ref{sec:c}, the following conjecture of Levit and Tankus \cite{levit} that attempts to characterize $\mathbf{W_2}$ graphs. 

\begin{conjecture}\label{conj}
A graph $G$ is $\mathbf{W_2}$ if and only if for every vertex $v$ in $G$ and every maximal independent set $S$ in $G - N[v]$, the largest independent set in $N(v) \setminus N(S)$ consists of a single vertex.
\end{conjecture}

We should remark that, in their same paper, they proved the conjecture under the extra assumption that $G$ is well-covered.

\section{The proof of Theorem \ref{thm:wk}}\label{sec:wk}

In this section, we prove Theorem \ref{thm:wk}. 

First, we require a definition. The \emph{lexicographic product} $G_1\cdot G_2$ of two vertex disjoint graphs $G_1$ and $G_2$ is the graph with vertex set $V(G_1)\times V(G_2)$ and edge set $\{(u,v)(x,y) : ux \in E(G_1) \mbox{ or } u = x \mbox{ and } vy \in E(G_2)\}$.

\begin{proof}[Proof of Theorem \ref{thm:wk}]
We construct a graph $H$ from an arbitrary graph $G$ in polynomial time such that $G$ belongs to $\mathbf{W_1}$ if and only if $H$ belongs to $\mathbf{W_k}$. Since recognizing $\mathbf{W_1}$ graphs is coNP-hard\cite{chvatal}, this will imply the theorem.

Let $H =  G \cdot K_k$ be the lexicographic product of $G$ and the complete graph $K_k$ on $k$ vertices.  For any $v\in \text{V}(G)$, let $K^v =   \{v\} \times \text{V}(K_k)$, and for  any $A \subseteq V(H)$, let $p_G(A) = \{u \in V(G): (u, v) \in A\}$ denote the \emph{projection of $A$ on $G$}.   We begin with a  simple claim that guarantees that any independent set in $H$ projects to an independent set in $G$ having the same cardinality.

\begin{claim}\label{claim:1}
For any independent set $A$ of $H$, $p_G(A)$ is an independent set of $G$ of size $|A|$.
\end{claim}
\begin{proof}
Suppose there are two different vertices $u_1$ and $u_2$ in $p_G(A)$ such that $u_1u_2\in E(G)$. By definition, $A$ contains vertices $(u_1, x)$ and $(u_2, y)$ for some (not necessarily distinct) $x, y \in V(K_k)$ and these must, by definition, be adjacent, which contradicts that $A$ is independent. Therefore, $p_G(A)$ is independent. 

Now, by definition $ |A| \geq |p_G(A)| $. To see that $|p_G(A)| = |A|$, note that since $A$ independent, if  $(u, v_1) \in A$  and $(u, v_2) \in A$ then $v_1v_2 \in E(K_k)$ implies $(u,v_1)(u,v_2) \in E(H)$, a contradiction. The claim is proved. 
\end{proof}

We further establish that the maximum independent sets of $G$ and $H$ have the same cardinality.
\begin{claim}\label{claim:2}
$\alpha(G)= \alpha(H)$.
\end{claim}

\begin{proof}
For any independent set $A$ in $G$ and vertex $u \in V(K_k)$, by definition, the set $A \times \{u\}$ forms an independent set in $H$; thus, $\alpha(H) \geqslant \alpha(G)$. And for any independent set $A$ in $H$,  the projection $p_G(A)$ is, by Claim \ref{claim:1}, an independent set in $G$ of cardinality $|A|$  and so $\alpha(G) \geqslant \alpha(H)$.
\end{proof}

We now show that $H$ is $\mathbf{W_k}$, assuming $G$ is $\mathbf{W_1}$.  Let $A_1, \ldots, A_k$ be disjoint independent vertex sets in $H$. For each $v\in V(G)$, we (re)label the vertices of $K^v$ as $a_1^v, \dots, a_k^v$ arbitrarily unless 
\begin{itemize}
\item[($\star$)] whenever $K^v \cap A_i \not= \emptyset$ for some $i \in [k]$, then the unique vertex $(v, u) \in K^v \cap A_i$ for some $u \in K_k$ is labelled as $a_i^v$. 
\end{itemize}

Now, by Claim \ref{claim:1}, $p_G(A_i)$ is, for $i \in [k]$, independent in $G$  and, as $G$ is $\mathbf{W_1}$, can be extended to a maximum independent set $T_i$ in $G$. It follows that the sets $\{ a_i^u \mid u \in T_i\}$ for $i \in [k]$ are independent in $H$ and of size $\alpha(G)$ and thus maximum by Claim \ref{claim:2} and, by ($\star$), pairwise disjoint. Thus, $H$ is $\mathbf{W_k}$.

To see that $G$ is $\mathbf{W_1}$, assuming $H$ is $\mathbf{W_k}$, consider any independent set $A$ in $G$, and fix some $u\in V(K_k)$. The set $A \times \{u\}$ forms an independent set in $H$ and, as $H$ is $\mathbf{W_k}$, extends to a maximum independent set $T$ in $H$. Consequently, $p_G(T)$ is an independent set in $G$ that contains $A$ and of maximum size size by Claim \ref{claim:2}. Thus, $G$ is $\mathbf{W_1}$. This concludes the proof.
\end{proof}

Recall that the class of $\mathbf{W_2}$ graphs is precisely the class of 1-well-covered graphs without any isolated vertices. Consequently 

\begin{corollary}
Recognizing $1$-well-covered graphs is coNP-hard.
\end{corollary}

\section{The proof of Theorem \ref{thm:FPT}}\label{sec:fpt}
In this section, we prove Theorem \ref{thm:FPT}.

 Before giving the details, let us briefly sketch the approach. We shall establish that testing whether a graph is well-covered and whether it is possible to extend a fixed number of pairwise disjoint independent sets in a graph into pairwise disjoint maximal independent sets is FPT in the clique-width of the graph. Since both these properties correspond to a characterization of $\mathbf{W_k}$ graphs (and this follows from the definitions), this will complete the proof.

\begin{proof}[Proof the Theorem \ref{thm:FPT}]
To prove the theorem, we express the problem in LinEMSOL and then apply Courcelle's theorem. In order to achieve this, we define the formula
\begin{equation*}
\mathbf{Wk} = \quad ~\forall_{X_1,\dots,X_k\subseteq V}
 ~\mathbf{indep}_k(X_1,\dots X_k)
 \end{equation*}
 which checks if $X_1, \dots, X_k$ are pairwise disjoint independent vertex subsets of $G$, and if so, whether there exist $k$ pairwise disjoint maximal independent sets $X'_1, \dots, X'_k$ such that $X_i \subseteq X'_i$ for $i \in [k]$. Here
\begin{align*}
\mathbf{indep}_k(X_1,\dots,X_k) &=  \left(\bigwedge_{1\leq i \leq k} \mathbf{indep}(X_i) ~\wedge \bigwedge_{1\leq i < j \leq k}\textbf{disjoint}(X_i,X_k)\right) \\ \\
& \Longrightarrow \left(\exists_{X_1',\dots, X_k'\subseteq V} \bigwedge_{1\leq i \leq k} \mathbf{subset}(X_i,X_i') \wedge ~\mathbf{maximal}(X_i') \right.\\
&\qquad \left.\wedge \bigwedge_{1\leq i <j \leq k} \mathbf{disjoint}(X_i',X_j') \right)
\end{align*}
 where $\mathbf{indep}(X) = \quad \forall_{u,v\in X} ~\neg \mathbf{adj}(u,v)$ checks whether $X$ is independent, 
$\textbf{maximal}(X) = \quad \mathbf{indep}(X) \wedge \forall_{v\notin X} \exists_{v\in X} \mathbf{adj}(u,v)$ checks whether $X$ is both maximal and independent,  
$\textbf{disjoint}(X,Y) = \quad \forall_{u\in X} \forall_{v\in Y} ~(u\neq v)$ checks whether $X$ and $Y$ are disjoint and, lastly, 
$\textbf{subset}(X,X') = \quad \forall_{x\in X} ~(x\in X')$ checks whether $X$ is a subset of $X'$. We have the following claim.

\begin{claim}\label{claim}
A graph $G$ is $\mathbf{W_k}$ if, and only if, $G$ is well-covered and $G\models  \mathbf{Wk}$.
\end{claim}

\begin{proof}
Suppose first that $G$ is $\mathbf{W_k}$. Then, by definition, $G$ is well-covered. Let $X_1,...,X_k$ be pairwise disjoint independent vertex subsets of $G$. Since $G$ is $\mathbf{W_k}$, there exist pairwise disjoint maximal (in fact, maximum) independent vertex subsets $X_1',\dots X_k'\subseteq V$ such that $X_i \subseteq X_i'$ for $i \in [k]$. Thus $G\models \mathbf{Wk}$.

Conversely, suppose that $G$ is well-covered and $G\models \mathbf{Wk}$. Let $X_1,\dots,X_k$ be $k$ pairwise disjoint independent vertex subsets in $G$. Since $G\models \mathbf{Wk}$, there are $k$ pairwise disjoint maximal independent sets $X_1',\dots,X_k'$ such that  $X_i \subseteq X_i'$ for $i \in [k]$. But $G$ being well-covered implies each $X_i'$ is also maximum and thus $G$ is $\mathbf{W_k}$ as needed. 
\end{proof}

To complete the proof: since the property of being a maximal independent set is MSO${}_1$ by our formulas, given an $n$-vertex graph $G$ of clique-width $w$, we can compute the size $\beta(G)$ of the smallest maximal independent set and the size $\alpha(G)$ of the largest maximal independent set in time $f_1(w).n^{\mathcal{O}(1)}$ by Courcelle's theorem on LinEMSOL logic, where $f_1$ is a computable function. In particular, since  a graph $G$ is well-covered iff $\alpha(G) = \beta(G)$, deciding whether $G$ is well-covered can thus be done in time $f_1(w).n^{\mathcal{O}(1)}$.

On the other hand, we can decide if $G\models \mathbf{Wk}$ in time $f_2(w).n^{O(1)}$ using Courcelle's theorem on MSO${}_1$ logic, where $f_2$ is a computable function. 

Therefore, together with Claim \ref{claim}, there is an algorithm that decides if $G$ is $\mathbf{W_k}$ in time $f(w).n^{O(1)}$, where $f = f_1 + f_2$ completing the proof. 
\end{proof}

\section{A counterexample}\label{sec:c}

In this section, we refute Conjecture \ref{conj} by constructing the following counterexample.  Recall that, for a positive integer $k$, a \textit{$k$-matching} is a disjoint union of $k$ edges. Given a vertex of a $k$-matching, we call its unique neighbor its \emph{sibling}. 

Now, let $s$ and $t$ be any positive integers, $s < t$.  Let $A$ consist of a $s$-matching and $B$ consist of a $t$-matching. Consider the graph $G$ formed from $A$ and $B$ by adding all possible edges between $A$ and $B$. We claim that $G$ is a counterexample to Conjecture \ref{conj}.

Note that $G$ is not well-covered since it contains a maximal independent of size $s$ and a maximal independent set of size $t$.  In particular, $G$ is not $\mathbf{W_2}$. However, for any vertex $v$ in $G$ and any maximal independent set $S$ of $G- N[v]$ the set $N(v) \setminus N(S)$ consists of exactly the sibling of $v$, as needed.

 \bigskip

\noindent
\textbf{Acknowledgements.} We thank the referee for valuable comments that improved the presentation and for spotting an inaccuracy. Carl Feghali was supported by the French National Research Agency under research grant ANR DIGRAPHS ANR-19-CE48-0013-01

 \bibliography{bibliography}{}
\bibliographystyle{abbrv}
 
\end{document}